\documentclass[12pt,twoside]{amsart}
\usepackage{amssymb,amsmath,amsfonts,amsthm}
\usepackage{mathrsfs}
\usepackage[X2,T1]{fontenc}
\usepackage[applemac]{inputenc}
\usepackage{cite}
\usepackage{latexsym}
\usepackage{verbatim}
\usepackage{soul}
\usepackage[dvipsnames]{xcolor}
\def\T{\mathbb{T}}

\def\R{\mathbb R}
\def\C{\mathbb C}
\def\N{\mathbb N}
\def\Z{\mathbb Z}
\def\Q{\mathbb Q}

\usepackage{colortbl}

\newcommand{\beqsn}{\arraycolsep1.5pt\begin{eqnarray*}}
\newcommand{\eeqsn}{\end{eqnarray*}\arraycolsep5pt}
\newcommand{\beqs}{\arraycolsep1.5pt\begin{eqnarray}}
\newcommand{\eeqs}{\end{eqnarray}\arraycolsep5pt}

\newtheorem{theorem}{Theorem}
\newtheorem{lemma}{Lemma}

\newtheorem{proposition}{Proposition}
\newtheorem{definition}{Definition}

\newtheorem{remark}{Remark}

\catcode`\@=11

\renewcommand{\section}%
   {\setcounter{equation}{0}\@startsection {section}{1}{\z@}{-3.5ex plus -1ex
  minus -.2ex}{2.3ex plus .2ex}{\Large\bf}}

\title[]{Globally solvable time-periodic evolution \\ equations in Gelfand-Shilov classes}

\author[F. Avila Silva]{Fernando de \'Avila Silva}
\address{Department of Mathematics  \\ Federal University of Paran\'a, \\ Curitiba, CEP 81531-980, Caixa Postal 19081, Paran\'a, Brazil}
\email{fernando.avila@ufpr.br}

\author[M. Cappiello]{Marco Cappiello}
\address{Dipartimento di Matematica ``G. Peano'' \\Universit\`a di Torino\\
	Via Carlo Alberto 10\\
	10123 Torino\\
	Italy}
\email{marco.cappiello@unito.it}
\begin{document}


\begin{abstract} In this paper we consider a class of evolution operators with coefficients depending on time and space variables $(t,x) \in \T \times \R^n$, where $\T$ is the one-dimensional torus and prove necessary and sufficient conditions for their global solvability in (time-periodic) Gelfand-Shilov spaces. The argument of the proof is based on a characterization of these spaces in terms of the eigenfunction expansions given by a fixed self-adjoint, globally elliptic differential operator on $\R^n$.  
\end{abstract}

\maketitle

\textit{keywords:} Gelfand-Shilov spaces, periodic equations, global solvability, Fourier analysis
\\


\maketitle

\section{Introduction}

Global solvability for evolution operators with periodic coefficients is a huge field of investigation which counts many contributions, see for instance \cite{AlbanesePopiv, BDGK15, BDG18,  BDG17, Hou79, HouCardoso77, PETRONILHO2005}.
In the most part of situations, this problem is strictly connected with the one of the global hypoellipticity. In the above mentioned papers, the operators under consideration have coefficients which are periodic with respect to both time and space variables $(t,x)$ or just with respect to $t$ and independent from $x$. Hence, the coefficients are defined on the torus (or on products of tori). More recently, the problem of solvability has been investigated also in other compact setting with special attention to Lie groups, e.g. \cite{Araujo2019,AFR,KMR,KMR1}.  In particular, we point out that an important tool, present in all these references and here, is a  Fourier analysis characterizing the functional spaces under investigation.

The aim of this paper is to discuss global solvability for operators of the form 
\begin{equation}\label{op-intro}
	L= D_t + c(t)P(x,D_x), 
\end{equation}
where $D_t = i^{-1}\partial_t$, the coefficient $c(t)=a(t)+ib(t)$ is complex-valued and belongs to some Gevrey class $\mathcal{G}^{\sigma}(\T)$, $\sigma >1$, and $P(x, D_x)$ is a self-adjoint differential operator  of type
\begin{equation}\label{P-intro}
	P = P(x,D) = \sum_{|\alpha| + |\beta| \leq m} c_{\alpha,\beta} x^{\beta} \partial_x^{\alpha}, \ c_{\alpha,\beta} \in \R,
\end{equation}
of order	$m\geq 2$, 	satisfying  the global ellipticity property 
\begin{equation}\label{P-elliptic}
	p_m(x,\xi) = \sum_{|\alpha| + |\beta| = m} c_{\alpha,\beta} x^{\beta} (i\xi)^{\alpha} \neq 0, \quad (x,\xi) \neq (0,0).
\end{equation}

Before treating in detail the evolution operator \eqref{op-intro}, let us consider the operator $P$ in \eqref{P-intro}. Self-adjointness and condition \eqref{P-elliptic} imply that $P$ has a discrete spectrum consisting in a sequence of real eigenvalues $\lambda_j$ such that $|\lambda_j| \to \infty$ for $j \to \infty$ and satisfying
\begin{equation}\label{weyl}
	|\lambda_j| \sim  \rho j^{\, m/2n}, \ \textrm{ as } \ j \to \infty,
\end{equation}
for some positive constant $\rho$. Moreover, the eigenfunctions of $P$ form an orthonormal basis of $L^2(\R^n)$. The most relevant example is the Harmonic oscillator $P(x,D)= |x|^2-\Delta, $ where $\Delta$ denotes the standard Laplace operator on $\R^n$.
Such operators and their pseudodifferential generalizations have been deeply studied on the Schwartz space $\mathscr{S}(\R^n)$ of smooth rapidly decreasing functions and on the dual space of tempered distributions $\mathscr{S}'(\R^n)$, cf. \cite{Shubin}.  More recently, however, it has been shown that a more appropriate functional setting for such operators is given by the so-called Gelfand-Shilov spaces of type $\mathscr{S}$, introduced in \cite{GS2, GS3} as an alternative setting to the Schwartz space for the study of partial differential equations. 

Given $\mu >0, \nu >0$, the Gelfand-Shilov space $\mathcal{S}^\mu_\nu(\R^n)$ is defined as the space of all $f \in C^\infty(\R^n)$ such that
\begin{equation*}
	\sup_{\alpha, \beta \in \N^n} \sup_{x \in \R^n} A^{-|\alpha+\beta|} \alpha!^{-\nu}\beta!^{-\mu}
	|x^\alpha \partial_x^\beta f(x)| <+\infty 
\end{equation*}
for some $A>0$, or equivalently,
\begin{equation*}
	\sup_{\beta \in \N^n} \sup_{x \in \R^n} C^{-|\beta|} \beta!^{-\mu}\exp( c|x|^{1/\nu})
	| \partial_x^\beta f(x)| <+\infty 
\end{equation*}
for some $C,c>0$.
Elements of $\mathcal{S}^\mu_\nu(\R^n)$ are then smooth functions presenting uniform analytic or Gevrey estimates on $\R^n$ and admitting an exponential decay at infinity. The elements of the dual space $(\mathcal{S}^\mu_\nu)'(\R^n)$ are commonly known as \textit{temperate ultradistributions}, cf. \cite{Pil}.

In the last two decades, Gelfand-Shilov spaces have become very popular in the study of microlocal and time-frequency analysis with many applications to partial differential equations, see for instance \cite{Arias_GS, AACJEE, ACJDE, ACJMPA, TSU, CGR1, CGR2, CGR3, CPP, CoNiRo1, CoNiRo2, NR, Prangoski} and the references quoted therein. 
Concerning in particular the operators in \eqref{P-intro}, \eqref{P-elliptic}, we mention the hypoellipticity results in \cite{CGR1, CGR2} which show that the solutions $u \in \mathcal{S}'(\R^n)$ of the equation $Pu=f \in \mathcal{S}_\nu^\mu(\R^n)$ actually belong to $\mathcal{S}_\nu^\mu(\R^n)$. In particular, the eigenfunctions of $P$ are in $\mathcal{S}^{1/2}_{1/2}(\R^n)$.  Recently, these spaces have been also characterized in terms of eigenfunction expansions, see \cite{CGPR, GPR}. 
\\
\indent
In the paper \cite{AC}, we introduced the  \textit{time-periodic Gelfand-Shilov spaces}  $\mathcal{S}_{\sigma,\mu}(\T \times \R^n)$ with $\sigma \geq 1, \mu \geq 1/2$, ($\mathcal{S}_{\sigma,\mu}$ in short), as the space of all smooth functions on $\T \times \R^n$ such that
\begin{equation}\label{firstnorm}
	|u|_{\sigma, \mu, C}:=
	\sup_{\alpha, \beta \in \N^n, \gamma \in \N}C^{-|\alpha+\beta|-\gamma}\gamma!^{-\sigma} (\alpha!\beta!)^{-\mu}\sup_{(t,x) \in \mathbb{T}\times \R^n} |x^\alpha \partial_x^\beta \partial_t^\gamma u(t,x)|
\end{equation}
is finite 	for some positive constant $C$, and we studied the global hypoellipticity of the operator $L$ in \eqref{op-intro} in this setting. Notice that the elements of $\mathcal{S}_{\sigma, \mu}(\T \times \R^n)$   belong to the symmetric Gelfand-Shilov spaces $\mathcal{S}_{\mu}^{\mu} (\R^n)$, cf. \cite{GS2, GS3}, with respect to the variable $x$, while are Gevrey regular and periodic in $t$.

In order to achieve our result we adapted to the periodic setting a characterization of classical Gelfand-Shilov spaces in terms of eigenfunction expansions proved in \cite{GPR}. Precisely, the orthonormal basis of eigenfunctions $\{\varphi_j \}_{j \in \N}$ of $P$ allows to write any $u \in \mathcal{S}_{\sigma,\mu}(\T \times \R^n)$ (respectively $u \in \mathcal{S}'_{\sigma,\mu}(\T \times \R^n)$ ) as the sum of a Fourier series 
\begin{equation*}
	u(t,x) = \sum_{j \in \N} u_j(t) \varphi_j(x),
\end{equation*} where $u_j(t)$ is a sequence of Gevrey functions (respectively distributions) on the torus satisfying suitable exponential estimates.
This allowed to discretize the equation $Lu=f$ and to apply 
the typical arguments of the analysis on the torus. The results proved in \cite{AC} will be also used in the present paper and for this reason they are briefly recalled in Section \ref{sec-notations}.
\\ \indent Let us now come to the main results of the paper. 
In order to introduce a suitable notion of global solvability for our problem, let us consider the space
$$
\mathscr{F}_\mu = \bigcup_{\sigma > 1}\mathcal{S}_{\sigma, \mu}.
$$ 
\begin{definition} \label{defglobalsolvability1}
	The operator $L$ is said to be $\mathcal{S}_{\mu}$-globally solvable if for every  $f \in \mathscr{F}_\mu$ there exists a solution $u \in \mathscr{F}_\mu$ of the equation $Lu=f$.
\end{definition}

As standard in this type of problems, the solvability of the operator is strictly conditioned by the behavior of the imaginary part $b(t)$ of the coefficient $c(t)$ and in particular by its sign changing. Namely, when the sign of $b$ is constant, then the global solvability of $L$ is equivalent to the fact that the average $c_0 = (2\pi)^{-1}\int_0^{2\pi} c(t)\, dt$ satifies some algebraic conditions, called \textit{Diophantine conditions}, whereas when $b$ changes sign the global solvability is related with the topological properties of the sets 
\begin{equation}\label{superlevel}
	\Omega_r= \left\{t \in \T: \int_0^t b(s)\, ds >r \right\}, \qquad r \in \R,
\end{equation} 
and with the size of set
\begin{equation}\label{Z}
	\mathcal{Z} =\{j \in \N; \ \lambda_jc_0 \in \Z\}.
\end{equation}

The main result of this paper reads as follows.

\vskip0.2cm
\begin{theorem}\label{main_theorem}
	Let $L$ be defined by \eqref{op-intro}, \eqref{P-intro}, with $P$ self-adjoint and satisfying \eqref{P-elliptic}. Then:
	\begin{enumerate}
		\item [(a)] if $b$ does not change sign, then $L$ is $\mathcal{S}_{\mu}$-globally solvable if and only if $c_0$ satisfies the following condition:
		\vskip0.2cm
		\noindent	$(\mathscr{A})$ \, for every $\epsilon>0$ there exists $C_{\epsilon}>0$ such that
		\begin{equation*}
			| \tau - c_0\lambda_j| \geqslant C_{\epsilon} \exp \left(-\epsilon j^{\frac{1}{2n\mu}}\right),
		\end{equation*}
		for all $(\tau, j) \in \Z\times \N$, such that $\tau - c_0\lambda_j \neq 0$.
		\item [(b)] if $b$ changes sign, then $L$ is $\mathcal{S}_{\mu}$-globally solvable if and only if $\mathcal{Z}^C=\N \setminus \mathcal{Z}$ is finite  and the sets $\Omega_r$ in \eqref{superlevel} are connected  for all $r \in \R$.
	\end{enumerate}
\end{theorem}

Notice that the connectedness conditions in (b) appears frequently in the study of global solvability on the torus, see for instance 
\cite{BDGK15,BDG18,BDG17,Treves}.

The paper is organized as follows.
In Section \ref{sec-notations} we introduce time-periodic Gelfand-Shilov spaces and their characterization in terms of eigenfunction expansions. Moreover, we discretize the equation $Lu=f$ reducing it to a family of ordinary differential equations involving the Fourier coefficients of $u$ and $f$. Finally, we recall the results obtained in \cite{AC} about global hypoellipticity. In Section \ref{sec-solvability} we make some preparation to the proof of Theorem \ref{main_theorem}. Namely, we introduce a space of admissible functions $f$ out of which global solvability cannot be obtained and relate it to the kernel of the transpose operator $^tL$. Moreover, we prove necessary and sufficient conditions for global solvability in the case when the coefficient $c(t)$ is constant. Then we start to treat the case of time depending coefficients and show that it is possible to reduce $L$ to a normal form via a suitable transformation. In Section \ref{Sec_proof_main_the} we prove Theorem \ref{main_theorem}. The proof consists in several steps and the strategy is the following: first, we verify the sufficiency part in  Theorem \ref{suffientt_cond}. The analysis of the necessity part is the focus of Subsection \ref{sec_nec_cond}. The \textit{algebraic} conditions are given by Propositions \ref{prop_a_0.b_0},  \ref{prop1-nec-par}, and \ref{prop_b_chang}. Finally,  the topological condition on $\Omega_{r}$ is verified  by Theorem \ref{Theorem_nec_Omega}.

\section{Notations and preliminary results  \label{sec-notations}} 
Let us start by recalling some basic properties of the spaces $\mathcal{S}_{\sigma, \mu}$ and $\mathcal{S}'_{\sigma, \mu}$.

\subsection{Time-periodic Gelfand-Shilov spaces and eigenfunction expansions}

Throughout the paper we denote by $\mathcal{G}^{\sigma,h}(\T)$ , 
$h>0$ and $\sigma \geq 1$, the space of all smooth functions $\varphi \in C^{\infty}(\T)$ such that there exists $C>0$ satisfying 
\begin{equation*}
	\sup_{t \in \T} |\partial^{k}\varphi(t)| \leq C h^{k}(k!)^{\sigma}, \ \forall k \in \Z_+.
\end{equation*}
Hence, $\mathcal{G}^{\sigma,h}(\T)$ is a Banach space endowed with the norm
\begin{equation*}
	\|\varphi\|_{\sigma,h} = \sup_{k \in \Z_+}\left \{\sup_{t \in \T} |\partial^{k}\varphi(t)|  h^{-k}(k!)^{-\sigma}\right\},
\end{equation*}
and  the space of periodic Gevrey functions of order $\sigma$ is defined by
$$
\mathcal{G}^{\sigma}(\T) =\displaystyle \underset{h\rightarrow +\infty}{\mbox{ind} \lim} \;\mathcal{G}^{\sigma, h} (\mathbb{T}),
$$
Its dual space will be denoted by $(\mathcal{G}^{\sigma})'(\T)$.

Similarly, fixed $\sigma >1, \mu\geq 1/2, C>0$ and denoting by $\mathcal{S}_{\sigma, \mu, C}$ the space of all smooth functions on $\T \times \R^n$ for which the norm \eqref{firstnorm} is finite, it is easy to prove that $\mathcal{S}_{\sigma, \mu, C}$ is a Banach space and we can endow $\mathcal{S}_{\sigma, \mu} = \bigcup\limits_{C> 0}
\mathcal{S}_{\sigma,\mu,C}$ with the inductive limit topology
$$
\mathcal{S}_{\sigma,\mu}  =\displaystyle \underset{C\rightarrow +\infty}{\mbox{ind} \lim} \;\mathcal{S}_{\sigma,\mu,C}.
$$
We shall then denote by
$\mathcal{S}_{\sigma, \mu}'$  the space of all linear continuous forms $u: \mathcal{S}_{\sigma, \mu} \to \C$.

\vskip0.2cm
\begin{remark}
	By \cite[Lemma 3.1]{GPR} an equivalent norm to \eqref{firstnorm} on $\mathcal{S}_{\sigma,\mu,C}$ is given by the following:
	\begin{equation}\label{secondnorm}
		\| u \|_{\sigma, \mu,C} := \sup_{\gamma, M \in \N} C^{-M-\gamma} M!^{-m\mu} \gamma!^{-\sigma} \| P^M \partial_t^\gamma u \|_{L^2(\T \times \R^n)}
	\end{equation}
	In order to take advantage of the properties of $P$, in the proof of Theorem \ref{Theorem_nec_Omega} we will use the norm \eqref{secondnorm} rather than \eqref{firstnorm}. 
\end{remark}

Now we want to recover the Fourier analysis presented in \cite{AC}. With this purpose we recall the characterization of $\mathcal{S}_{\sigma, \mu}$ and $\mathcal{S}'_{\sigma, \mu}$ in terms of eigenfunction expansions. For this, let $\varphi_j \in \mathcal{S}_{1/2}^{1/2}(\R^n), j \in \N$, be the eigenfunctions of the operator $P$ in \eqref{P-intro}. We have the following results.

\vskip0.2cm
\begin{theorem}
	Let $\mu\geq 1/2$ and $\sigma \geq 1$ and let $u \in \mathcal{S}_{\sigma,\mu}'$. Then $u \in \mathcal{S}_{\sigma,\mu}$ if and only if 	it can be represented as 
	\begin{equation*}
		u(t,x) = \sum_{j \in \N} u_j(t) \varphi_j(x),
	\end{equation*}
	where
	$$
	u_j(t) = \int_{\R^n} u(t,x)\varphi_j(x)dx,
	$$
	and there exist $C >0$ and $\epsilon>0$ such that
	\begin{equation} \label{deccoeff}
		\sup_{t \in \T} | \partial_t^k u_j(t)| \leq
		C^{k+1} (k!)^{\sigma} \exp \left[-\epsilon j^{\frac{1}{2n\mu}} \right] \ \forall j,k \in \N.
	\end{equation}		
	
\end{theorem}
\begin{proof}
	See	\cite[Theorem 2.4]{AC}.
\end{proof}

\begin{proposition}
	Let $\{u_j\}_{j \in \N} \subset \mathcal{G}^{\sigma}(\T)$ be 
	a sequence such that  for any $\epsilon>0$, there exists $C_{\epsilon}>0$ such that
	\begin{equation*}
		\sup_{t \in \T} | u_j(t)|\leq C_{\epsilon} 
		\exp\left(\epsilon j^{\frac{1}{2n\mu}}\right), \ \forall j \in \N.
	\end{equation*}
	
	Then, 
	\begin{equation*}
		u(t,x) = \sum_{j \in \N} u_j(t) \varphi_j(x)
	\end{equation*}
	belongs to $\mathcal{S}_{\sigma,\mu}'$ and
	$$
	\langle  u_j \, , \, \psi(t) \rangle  = \langle  u \, , \, \psi(t)\varphi_j(x) \rangle, \ \forall \psi \in  \mathcal{G}^{\sigma}(\T).
	$$
	
	We use the notation $\{u_{j}\}  \rightsquigarrow u \in  \mathcal{S}_{\sigma,\mu}$.			
\end{proposition}

\begin{proof}
	See	\cite[Lemma 2.7 and Theorem 2.9]{AC}.
\end{proof}

\subsection{Discretization of equation $Lu=f$ and global hypoellipticity}

In this Subsection, we apply Fourier expansions in the equation $Lu=f$ and recall the results on global hypoellipticity proved in \cite{AC}. To do this,   consider the space

$$
\mathscr{U}_\mu = \bigcup_{\sigma > 1}\mathcal{S}_{\sigma, \mu}',
$$ 
and let  $u \in \mathscr{U}_\mu$ be a solution of equation $Lu = f\in \mathcal{S}_{\sigma,\mu}$. By using 
\begin{equation*}
	u(t,x) = \sum_{j \in \N} u_j(t) \varphi_j(x)
	\ \textrm{ and } \
	f(t,x) = \sum_{j \in \N} f_j(t) \varphi_j(x),
\end{equation*}
we get that $Lu=f$ if and only if 
\begin{equation}\label{diffe-equations}
	\partial_t u_j(t) + i \lambda_j c(t) u_j(t) = if_j(t), \ t \in \T, \ j \in \N.
\end{equation}
The last equations can be solved by elementary methods by
\begin{equation}\label{fator_integrante}
	u_j(t) = \xi_j\exp\left( -i \lambda_j\int_{0}^{t} c(r) dr \right) +  i\int_{0}^{t}\exp\left(i\lambda_j \int_{t}^{s}c(r)dr \right) f_j(s)ds,
\end{equation}
for some $\xi_j \in \C.$ From the ellipticity of equation \eqref{diffe-equations} we get
$u_j \in\mathcal{G}^{\sigma}(\mathbb{T}),$ for all $j \in \N$. 

Now, let $\mathcal{Z}$ the set defined in \eqref{Z} and put
$c_0 = a_0 + ib_0$.  By the periodicity condition $u_j(0) = u_j(2\pi)$ we have the following:

\vskip0.2cm
\begin{lemma}\label{lemma-comp-con}
	If $j \in \mathcal{Z}$ and $Lu = f\in \mathcal{S}_{\sigma,\mu}$, then 
	\begin{equation}\label{admissible_condition}
		\int_{0}^{2\pi}\exp\left(i\lambda_j \int_{0}^{t}c(s)ds\right) f_j(t)dt = 0.
	\end{equation}
	In particular,  
	\begin{equation}\label{Sol-6}
		u_j(t) = \int_{0}^{t}\exp\left(i\lambda_j \int_{t}^{s}c(r)dr \right) f_j(s)ds
	\end{equation}
	is a solution of \eqref{diffe-equations}.
\end{lemma}

If $j\notin \mathcal{Z}$ equations \eqref{diffe-equations} have a unique solution, which can be written in the following equivalent two ways:
\begin{align}\label{Solu-1}
	u_j(t) = \frac{i}{1 - e^{-  2 \pi i\lambda_j c_0}} \int_{0}^{2\pi}\exp\left(-i\lambda_j\int_{t-s}^{t}c(r) \, dr\right) f_j(t-s)ds, 
\end{align}
or
\begin{align}\label{Solu-2}
	u_j(t) = \frac{i}{e^{ 2 \pi i \lambda_j c_0} - 1} \int_{0}^{2\pi}\exp\left(i\lambda_j\int_{t}^{t+s}c(r) \, dr\right) f_j(t+s)ds. 
\end{align}

Using formulas \eqref{Solu-1} and \eqref{Solu-2} in \cite{AC} we proved necessary and sufficient conditions for global hypoellipticity. Recall that a differential operator $\mathcal{Q}$ on $\T \times \R^n$  is said to be $\mathcal{S}_{\mu}$-globally hypoelliptic  if conditions $u \in \mathscr{U}_\mu$ and  $\mathcal{Q}u \in \mathscr{F}_{\mu}$ imply $u \in \mathscr{F}_\mu$.
Also in this case Diophantine conditions appear naturally to control the behavior of the sequences  
\begin{equation}\label{Theta_j-Gamma_j}
	\Theta_j = |1 - e^{-  2 \pi i c_0\lambda_j}|^{-1}
	\ \textrm{ and} \
	\Gamma_j = |e^{2 \pi i c_0\lambda_j} - 1|^{-1}.
\end{equation}
Namely, inspired by reference \cite{BDG18}, let us set the following condition for a complex number $\omega$:
\begin{description}	\item[($\mathscr{B}$)] for every $\epsilon>0$ there exists $C_{\epsilon}>0$ such that
	\begin{equation*}
		| \tau - \omega \lambda_j| \geqslant C_{\epsilon} \exp \left(-\epsilon j^{\frac{1}{2n\mu}}\right),
	\end{equation*}
	for all $(\tau, j) \in \Z\times \N$.
	
\end{description}

Notice that $(\mathscr{B})$ implies condition $(\mathscr{A})$ appearing in Theorem \ref{main_theorem}.	
The behavior of sequences  in \eqref{Theta_j-Gamma_j} and 
the condition $(\mathscr{B})$ can be connected in view of the following Lemma whose proof can be obtained by a slight modification of the arguments in the proof of  Lemma 2.5 in \cite{BDG18}. We leave the details to the reader. 

\vskip0.2cm
\begin{lemma}\label{lt2} Consider $\eta \geq1$  and $\omega \in \C.$  The following two conditions are equivalent:
	\begin{itemize}
		\item[i)] for each $\epsilon>0$ there exists a positive constant $C_\epsilon$ such that
		$$
		|\tau - \omega\lambda_j|\geqslant C_\epsilon\exp\{-\epsilon(|\tau|+j)^{1/\eta}\}, \ \forall  \tau \in \mathbb{Z},\,  \forall j \in \mathbb{N}.
		$$
		
		\item[ii)] for each $\delta>0$ there exists a positive constant $C_\delta$ such that 
		$$
		|1-e^{2\pi i \omega\lambda_j}|\geqslant C_\delta\exp\{-\delta j^{1/\eta}\}, \ \forall  j \in\mathbb{N}.
		$$
		
	\end{itemize}
\end{lemma} 

We can now recall the global hypoellipticity results proved in \cite{AC} concerning the case when the coefficient $c(t)$ is constant or depending on $t$ respectively.

\vskip0.2cm
\begin{theorem}\label{GH-time-ind}
	Operator
	\begin{equation*}
		\mathcal{L} = D_t + (\alpha + i\beta)P(x,D_x),  \ \alpha, \beta \in \R,
	\end{equation*}
	is $\mathcal{S}_{\mu}$-globally hypoelliptic if and only if one of the following conditions holds:
	\begin{enumerate}
		\item [(a)] $\beta \neq 0$;
		
		\item [(b)] $\beta =0$ and $\alpha$ satisfies condition 
		($\mathscr{B}$), or equivalently, for every $\epsilon>0$ there exists $C_{\epsilon}>0$ such that
		\begin{equation*}
			\inf_{\tau \in \Z} | \tau - \alpha\lambda_j| \geqslant C_{\epsilon} \exp \left(-\epsilon j^{\frac{1}{2n\mu}}\right), \ \textrm{ as }  \ j \to \infty.
		\end{equation*}
	\end{enumerate}
\end{theorem}
\begin{proof}
	See  \cite[Theorem 3.6]{AC}.
\end{proof}

\begin{theorem}\label{GH-time-dep}
	Operator
	$
	L = D_t + c(t)P(x,D_x)
	$
	is $\mathcal{S}_{\mu}$-globally hypoelliptic if and only if one of the following conditions holds:
	\begin{enumerate}
		\item [(a)] $b$ is not identically zero and does not change sign;
		
		\item [(b)] $b \equiv 0$ and $a_0$ satisfies condition 
		($\mathscr{B}$).
	\end{enumerate}
	
\end{theorem}
\begin{proof}
	See  \cite[Theorem 3.11]{AC}.
\end{proof}

\section{Global solvability \label{sec-solvability}}

In this section, we start  the study of global solvability and make some preliminary steps to the proof of Theorem \ref{main_theorem}. First, we observe that in view of Lemma \ref{lemma-comp-con} it is necessary to introduce a class of admissible functions of operator $L$, namely, the space $\mathscr{E}_{L,\mu}$ of all $f \in \mathcal{F}_{\mu}$ such that
$$
\int_{0}^{2\pi}\exp\left(i\lambda_j \int_{0}^{t}c(s)ds\right) f_j(t)dt = 0,
$$
whenever $j\in \mathcal{Z}.$ Therefore, we can refine the notion of solvability given in Definition \ref{defglobalsolvability1} as follows.

\vskip0.2cm
\begin{definition}\label{def_GS}
	We say that operator $L$ is $\mathcal{S}_{\mu}$-globally solvable if for every $f\in \mathscr{E}_{L,\mu}$ there exists $u \in \mathscr{F}_{\mu}$ such that $Lu=f$.
\end{definition} 

We observe that the solvability of operator $L$ is strongly connected with properties of its transpose $^tL$, cf. \cite{BDG17}.  We recall that $P$ is self-adjoint, with constant real coefficients, implying $^tP= P$ in view of
$$
^tP u = \overline{P^*(\overline{u})} = \overline{P(\overline{u})} = Pu.
$$

Therefore, 
$$
^tL= - D_t + c(t)P(x,D_x),
$$	
and if  $f = Lu$, for some $u \in \mathscr{F}_{\mu}$, and $v \in \mbox{ker}(^tL)$ we get
$$
\langle v, f \rangle = \langle v, Lu \rangle = \langle ^tL v,u \rangle = 0,
$$
and 
$$
L(\mathcal{S}_{\sigma,\mu}) \subset [\mbox{ker}(^tL)]^{\circ} \doteq \{\phi \in \mathcal{F}_{\mu}: \langle \omega,\phi\rangle  = 0, \ \forall \omega \in \mbox{ker}(^tL) \}
$$

In particular, we may 	characterize $\mbox{ker}(^tL)$ in terms of Fourier coefficients as follows.

\vskip0.2cm
\begin{lemma}\label{charac_kern}
	We have $\omega \in \mbox{ker}(^tL)$ if and only if
	\begin{equation}
		\omega_j(t) = \left\{
		\begin{array}{l}
			0, \ j \notin \mathcal{Z},  \\
			\eta_j \exp\left(i\lambda_{j}\int_{0}^{t}c(s)ds\right), j \in \mathcal{Z},
		\end{array}
		\right.
	\end{equation}
	for some $\eta_j \in \C$. In particular, $\mathscr{E}_{L,\mu} = [\mbox{ker}(^tL)]^{\circ}$.

\end{lemma}
\begin{proof}
	Note that $\omega \in \mbox{ker}(^tL)$  if and only if
	$$
	\exp\left(-i\lambda_{j}\int_{0}^{t}c(s)ds\right) \omega_j(t) = \eta_j
	$$
	where $\eta_j \in \C$ satisfies the condition
	\begin{equation*}
		\eta_j \left[1 - \exp\left(i\lambda_j 2 \pi c_0\right)\right]=0,
	\end{equation*}
	since  $\omega_j(t)$ is $2\pi$-periodic.
	If $j \notin \mathcal{Z}$, then $\eta_j = 0$ and $\omega_j(t) \equiv 0$. On the other hand, for  $j \in \mathcal{Z}$, $\eta_j$ can be chosen arbitrarily.

	Now, given $\phi \in  \mathscr{F}_\mu$ and $\omega = \sum_{j \in \N} \omega_j(t) \varphi_j(x) \in \mbox{ker}(^tL)$ we obtain
	\begin{equation} \label{<>}
		<\omega, \phi> = \sum_{k \in \mathcal{Z}} \left\{\eta_k \int_{0}^{2\pi} \exp\left(i\lambda_{k}\int_{0}^{t}c(s)ds\right)\phi_k(t) dt \right\}.
	\end{equation}

	By definition, if $\phi \in \mathscr{E}_{L,\mu}$, then 
	\begin{equation*}\label{cond_phi}
		\int_{0}^{2\pi} \exp\left(i\lambda_{k}\int_{0}^{t}c(s)ds\right)\phi_k(t) dt = 0, \forall k \in \mathcal{Z}.
	\end{equation*}
	which implies $<\omega, \phi> = 0$, then $\phi\in [\mbox{ker}(^tL)]^{\circ}$.
	
	Conversely, if $\phi\in [\mbox{ker}(^tL)]^{\circ}$, then, fixed  $\ell \in \mathcal{Z}$, we can define a function $\omega^{\ell} \in \mbox{ker}(^tL)$ by setting
	$$
	\omega^{\ell}_k(t) = 
	\left\{
	\begin{array}{l}
		0, \ \textrm{ if } \ k \neq \ell, \\
		\exp\left(i\lambda_{\ell}\int_{0}^{t}c(s)ds\right), \ \textrm{ if } \ k=\ell.
	\end{array}
	\right. 
	$$
	
	Hence, from \eqref{<>} we obtain
	$$
	0 = <\omega^{\ell} , \phi>  = \int_{0}^{2\pi} \exp\left(i\lambda_{\ell}\int_{0}^{t}c(s)ds\right)\phi_\ell(t) dt,
	$$
	implying $\phi \in \mathcal{E}_{L,\mu}.$
\end{proof}

As for global hypoellipticity, in order to prove Theorem \ref{main_theorem} we shall treat separately the case when the coefficient $c$ in \eqref{op-intro} is constant and the one when it is time depending. However, first we state the following general fact.

\vskip0.2cm
\begin{proposition}\label{GHimplGS}
	If $L$ is $\mathcal{S}_{\mu}$-globally hypoelliptic, then it is $\mathcal{S}_{\mu}$-globally solvable.
\end{proposition}

\begin{proof} 
	It follows from Theorem \ref{GH-time-dep} that either $b \equiv 0$ and $a_0$ satisfies condition $(\mathscr{B})$ or $b$ does not change sign, then $b_0 \neq 0$. In both cases the set $\mathcal{Z}$ is finite, cf. \cite[Theorem 3.14 and Corollary 3.9]{AC}. Moreover, by  the equivalency of expressions \eqref{Solu-1} and \eqref{Solu-2} we can admit $b(t)\geq 0$ without loss of generality.
	
	Now, for any $f  \in \mathscr{E}_{L,\mu}$ we may assume that $\{f_j\}_{j \in \N} \subset \mathcal{G}^{\sigma}(\T)$. If $j \in \mathcal{Z}$ we define  $u_j(t)$ by expression \eqref{Sol-6}, while in case $j \notin \mathcal{Z}$ we choose $u_j(t)$ as in \eqref{Solu-1}. Therefore, $u_j(t) \in \mathcal{G}^{\sigma}(\T)$ for all $j$ and 
	$$
	\partial_t u_j(t) +    i\lambda_j c(t) u_j(t) = if_j(t),  \ t \in \mathbb{T}.
	$$ 
	
	Since $\mathcal{Z}$ is finite, then estimates for $u_j(t)$ in the case $j \in \mathcal{Z}$ have no influence. On the other hand, for $j \notin \mathcal{Z}$, by a similar argument as in the proof of \cite[Theorem 3.6]{AC} (for $b_0 = 0$) and \cite[Theorem 3.12]{AC} (for $b(t)\not \equiv 0$) we obtain that  $\{u_{j}\}  \rightsquigarrow u \in   \mathscr{F}_{\mu}$ and $Lu=f$.
	
\end{proof}

\subsection{Time independent coefficients}\label{Sec_time_ind_coef}

In this subsection, we consider the time independent coefficients operator 
\begin{equation*} 
	\mathcal{L} = D_t + (\alpha + i\beta)P(x,D_x), \ t \in \T, \ \alpha, \beta \in \R.
\end{equation*}	
Note that  $(\alpha + i\beta)\lambda_{j} \in \Z$ if and only if $\beta =0$ and $\alpha \lambda_j \in \Z.$ In this case, $\mathscr{E}_{\mathcal{L},\mu}$ is given by all functions $f \in \mathcal{S}_{\mu}$ such that
$$
\int_{0}^{2\pi}\exp\left(i\lambda_j \alpha t \right) f_j(t)dt = 0,  \ \forall j\in \mathcal{Z}.
$$

The following standard formula will be useful in the sequel.
\begin{lemma}\label{lemma-exp-j}
	Let $s, p$ be positive numbers and $\tau \in \Z_+$. For every  $\eta>0$ there exist $C_{\eta}>0$ such that
	\begin{equation*}
		\gamma^{\tau p} \exp\left(-\eta \gamma^{1/s}\right) \leq   C_{\eta}^{\tau} (\tau!)^{s p}, \ \forall  \gamma \in \N.
	\end{equation*}
\end{lemma}

\begin{theorem}\label{GS-time-ind}
	Operator $\mathcal{L}$ 	is $\mathcal{S}_{\mu}$-globally solvable if and only if either $\beta \neq 0$ or $\beta =0$ and $\alpha$ satisfies condition  ($\mathscr{A}$).
\end{theorem}
\begin{proof}
	If $\beta \neq 0$, then the solvability is a consequence of Theorem \ref{GH-time-ind} and Proposition \ref{GHimplGS}. On the other hand, suppose that $\beta = 0$ and 
	assume condition  ($\mathscr{A}$). 
	
	Let $f \in \mathscr{E}_{\mathcal{L},\mu}$ be fixed. If  $j \in \mathcal{Z}$, we set 
	\begin{equation*}
		u_j(t) = \exp(-i\lambda_j \alpha t)\int_{0}^{t}\exp\left(i\lambda_j \alpha s\right) f_j(s)ds,
	\end{equation*}
	and
	\begin{align}\label{Sol-beta=0-jnotinZ}
		u_j(t) = \frac{i}{1 - e^{-  2  \pi i\lambda_j \alpha  }} \int_{0}^{2\pi}\exp\left(-i\lambda_j \alpha s \right) f_j(t-s)ds,
	\end{align}
	if $j \notin \mathcal{Z}$.
	
	In the first case, it follows by  Leibniz formula that
	\begin{multline*}
		\partial_t^\gamma u_j(t)= (-i\lambda_j \alpha)^\gamma\exp\left(-i\lambda_j \alpha t\right) \int_0^t \exp\left(i\lambda_j \alpha s\right) f_j(s)ds \\  + \sum_{0 \neq \delta \leq \gamma}\binom{\gamma}{\delta}(-i\lambda_j \alpha)^{\gamma-\delta}\sum_{\beta \leq \delta-1} \binom{\delta-1-\beta}{\beta} (i\lambda_j \alpha)^{\beta}\partial_t^{\delta-1-\beta}f_j(t),
	\end{multline*}
	and, by $|\lambda_j| \leq C'j^{m/2n},$ we get
	\begin{multline*}
		|\partial_t^\gamma u_j(t)| \leq C(C'|\alpha|)^{\gamma} j^{\frac{m|\gamma|}{2n}} \exp(-\varepsilon_0 j^{\frac{1}{2n\mu}}) \,  \\
		+ \sum_{\delta,\gamma,\beta} (C'|\alpha|)^{\gamma-\delta+\beta} j^{\frac{m(\gamma-\delta+\beta)}{2n}}  C^{\delta-\beta}(\delta-1-\beta)!^\sigma \exp(-\varepsilon_0 j^{\frac{1}{2n\mu}}),
	\end{multline*}
	where 
	$$
	\sum_{\delta,\gamma,\beta} =  \sum_{0 \neq \delta \leq \gamma}\binom{\gamma}{\delta}\sum_{\beta \leq \delta-1} \binom{\delta-1-\beta}{\beta}.
	$$

	The last estimate, Lemma \ref{lemma-exp-j} and standard factorial inequalities guarantee that $u_j \in \mathcal{G}^\sigma(\mathbb{T})$ and
	$$|\partial_t^\gamma u_j(t)| \leq C^{\gamma+1} \gamma!^{\textrm{max}\{\sigma, m\mu \}}\exp\left(-\frac{\varepsilon_0}{2} j^{\frac{1}{2n\mu}}\right)$$ for some positive constant $C$ independent of $\gamma.$
	Similarly, in case $j \notin \mathcal{Z}$, 	using Fa\`a di Bruno formula, we obtain the same type of estimate for \eqref{Sol-beta=0-jnotinZ}. Therefore, $\{u_{j}\}  \rightsquigarrow u \in  \mathcal{F}_\mu$ and $\mathcal{L} u = f$, which imply the solvability.

	Conversely,  assume that $\alpha + i\beta$ does not satisfy condition  ($\mathscr{A}$). By Lemma \ref{lt2} there are $\epsilon_0>0$ and a sequence $(j_\ell, \tau_\ell) \in \N \times \Z$ such that $|j_\ell|+ |\tau_\ell|$ is increasing and 
	\begin{equation*}
		0<|\tau_{\ell} - (\alpha + i \beta)\lambda_{j_\ell}|< 
		\exp \left(-\epsilon_0/2(|\tau_{\ell}|+j_{\ell})^{1/2n\mu}\right).
	\end{equation*}
	
	Since  $j_{\ell} \notin \mathcal{Z}$, for all $\ell$, then
	\begin{equation*}
		f_j(t) = \left\{
		\begin{array}{l}
			0, \ j \neq j_{\ell},  \\
			\exp\left[-\dfrac{\epsilon_0}{2}(|\tau_{\ell}|+j_{\ell})^{1/2n\mu} + i\tau_{\ell}t\right], \ j =j_{\ell}
		\end{array}
		\right.
	\end{equation*}
	is such that $\{f_{j}\}  \rightsquigarrow f \in   \mathscr{E}_{\mathcal{L},\mu}$. Therefore, if $u \in \mathscr{F}_{\mu}$ satisfies $\mathcal{L}u=f$ we should have 
	$$
	u_{j_\ell}(t) = \frac{i}{1 - e^{-  2  \pi i\lambda_{j_\ell} (\alpha + i \beta)  }} \int_{0}^{2\pi}\exp\left(-i\lambda_{j_\ell} (\alpha + i \beta) s \right) f_{j_\ell}(t-s)ds, 
	$$
	implying
	$$
	|u_{j_\ell}(0)| = \dfrac{\exp\left[-\dfrac{\epsilon_0}{2}(|\tau_{\ell}|+j_{\ell})^{1/2n\mu}\right]}{|\tau_{\ell} - (\alpha + i \beta)\lambda_{j_\ell}|} > 1, \forall \ell \in \N,
	$$
	which contradicts \eqref{deccoeff}.

\end{proof}

\subsection{Application: Cauchy problem \label{Sec_exemple_time_ind}}
Our results can be applied also to the problem of the existence and uniqueness of periodic solutions to the Cauchy problem associated to the operator $L$ in \eqref{op-intro}. We shall not give an exhaustive analysis of this problem but we shall limit to outline some examples. 
\\

\noindent
\textbf{Example 1.} Consider the operator $\mathcal{L} = D_t + (\alpha + i \beta) H$ and the Cauchy problem
\begin{equation}\label{harmonic-R-exa}
	\left\{
	\begin{array}{l}
		D_tu + (\alpha + i \beta) Hu  = f,\\
		u(0,x) = g(x)
	\end{array}
	\right.	
\end{equation}	
defined on $\T \times \R$, where $H$ stands for the Harmonic oscillator 
\begin{equation*}
	H = -\dfrac{d^2}{dx^2} + x^2, \ x \in \R,
\end{equation*}	
for which  $\lambda_j = 2j + 1$, $j \in \N_0$. 

If $\beta \neq 0$, then $\mathcal{L}$ is globally solvable. In view of the Fourier expansions $g(x) = \sum_{j \in \N} g_j \varphi_j(x)$, we get  
$$
g_j = u_j(0) = u_j(2 \pi), \ j \in \N.
$$

In the homogeneous case $f \equiv 0$, we have
$$
u_j(t) = g_j \exp\left[-i (2j + 1)(\alpha+ i\beta)  t \right]
$$
with
$$
g_j (1 - \exp\left[-i 2\pi (2j + 1) (\alpha+ i\beta) \right]) = 0.
$$
Since $\beta \neq 0$, then $g_j = 0$ for all $j$ and consequently $u \equiv 0$ is the unique solution of \eqref{harmonic-R-exa}, with $g\equiv 0$. If $g \neq 0$ then the Cauchy problem \eqref{harmonic-R-exa} has no solutions. Similar conclusion holds if $\beta = 0$ and $\alpha \notin \Q$.

Consider $\beta = 0$, $\alpha = 1/3$. In this case
$$
| \tau - \alpha \lambda_j| \geq \dfrac{1}{3}
$$
whenever  $\tau - 1/3(2j+1) \neq 0$. Then, condition ($\mathscr{A}$) is fulfilled and $\mathcal{L}$ is globally solvable. In particular, 
\begin{equation*}
	u_j(t) =	\left\{
	\begin{array}{l}
		g_j \exp\left(-i  \kappa_j  t \right), \ \textrm{ if } \  (2j+1) \in 3\N, \\
		0, \ \textrm{ if } \  (2j+1) \notin 3\N, 
	\end{array}
	\right.	
\end{equation*}	
where $\kappa_j = (2j+1)/3 \in \N$, generate the unique solution of \eqref{harmonic-R-exa}.

Now, let us consider the non-homogeneous case. Assuming  $\beta \geq 0$, if $\mathcal{L}$ is globally solvable, then we have
$$
u_j(t) = \exp\left( -i \alpha t \right)\left[ g_j +  \int_{0}^{t}\exp\left(i\lambda_j  \alpha s\right) f_j(s)ds\right]
$$
if $j \in \mathcal{Z}$ and 
$$
u_j(t) = \dfrac{i}{e^{ 2 \pi i \lambda_j (\alpha+ i\beta)} - 1} \int_{0}^{2\pi}\exp\left(i\lambda_j(\alpha+ i\beta)s\right) f_j(t+s)ds,
$$
whenever $j \notin \mathcal{Z}$, then we must have
$$
g_j =  \dfrac{i}{e^{ 2 \pi i \lambda_j (\alpha+ i\beta)} - 1} \int_{0}^{2\pi}\exp\left(i\lambda_j(\alpha+ i\beta)s\right) f_j(s)ds.
$$
The latter condition can be viewed as a compatibility condition between $f$ and the initial datum $g$. 

We observe that if $v=v(t,x)$ is a solution of the non-homogeneous problem with the initial datum $h(x)=\sum_{j \in \N}h_j\varphi_j(x)$. Then, 
$$
\|g-h\|_{L^2(\R)}^2 = \sum_{j \in \N}|g_j - h_j|^2
$$
and
$$
|g_j(t) - h_j(t)|= 
\left\{
\begin{array}{l}
	0, \ j \notin \mathcal{Z}, \\
	|g_j - h_j|, \  j \in \mathcal{Z}.
\end{array}
\right.
$$

\noindent
\textbf{Example 2.} Let us now consider 
$L = D_t + (\sin(t) + \cos(t)) H$ and the problem
\begin{equation*}
	\left\{
	\begin{array}{l}
		\partial_tu + i(\sin(t) + \cos(t)) Hu  = 0,\\
		u(0,x) = g(x)
	\end{array}
	\right.	
\end{equation*}	
defined on $\T \times \R$.

In this case, $a_0 = 1$ and $\mathcal{Z} = \N$. Since
$$
|\tau - \lambda_{j}a_0| > 1, 
$$
whenever $\tau - \lambda_{j}a_0 \neq 0$ we see that $L$ is globally solvable. Moreover, 
$$
u_j(t) = g_j \exp\left[-i\lambda_{j}\left(\sin(t) - \cos(t) + 1\right)\right].
$$

\subsection{Reduction to the normal form} \label{sec_reduction}

In this subsection, we show that  operator $L$ is globally solvable if and only if the same occurs to its  normal form, namely, the operator
$$
L_{a_0} = D_t + (a_0 + ib(t))P(x,D_x).
$$ 
This is a consequence of a conjugation formula presented in the next Proposition \ref{Prop-conjugation}.

\vskip0.2cm
\begin{proposition}\label{Prop-conjugation}
	There exists an isomorphism $\Psi: \mathscr{F}_\mu \to \mathscr{F}_\mu$ such that 
	\begin{equation}\label{conjugationformula}
		\Psi^{-1} \circ L\circ \Psi =  L_{a_0}.
	\end{equation}
	
\end{proposition}

\begin{proof}
	For each  
	$u = \sum_{j \in \N} u_j(t) \varphi_j(x) \in \mathscr{F}_\mu$ we define
	\begin{equation}\label{conjugation}
		\Psi u = \sum_{j \in \N} u_j(t) \exp(-i\lambda_j A(t))\varphi_j(x)
	\end{equation}
	and 
	\begin{equation*}
		\Psi^{-1} u = \sum_{j \in \N} u_j(t) \exp(i\lambda_j A(t))\varphi_j(x),
	\end{equation*}	
	where $A(t) = \int_{0}^{t}a(s)ds - a_0t$.
	
	If $\Psi, \Psi^{-1}: \mathcal {F}_\mu \to \mathcal {F}_\mu$ are well defined, then it is easy to verify linearity and the equality \eqref{conjugationformula}. Therefore, it is enough to prove that $\Psi u, \Psi^{-1} u \in \mathscr{F}_\mu$. For this, let $u \in \mathcal{S}_{\theta,\mu}$ and set
	\begin{equation*}
		\psi_j(t) =  u_j(t) \exp(-i\lambda_j A(t)), j \in \N.
	\end{equation*}
	
	Let $\gamma \in \Z_+$ be fixed. By Leibniz and Fa\`a di Bruno formulas we get
	\begin{eqnarray*}
		|\partial_t^{\gamma} \psi_j(t)| 
		& \leq & 2 \pi C_1 \sum_{\ell=0}^{\gamma}\left\{ \binom{\gamma}{\ell}
		C_2^{\ell} \sum_{\Delta(k), \, \ell}\frac1{k!} \cdot \frac{\ell!}{\ell_1! \cdots \ell_k! }j^{\frac{km}{2n}} C^{\ell-k+1}_3[(\ell -k)!]^\sigma \right. 
		\\ && \times \left.  \sup_{\tau \in [0,2\pi]}
		|\partial_t^{\gamma - \ell} u_j(\tau)| \right\}, 
	\end{eqnarray*}
	since $|\lambda_{j}| \leq C_4 j^{m/2n}$ by \eqref{weyl}, where
	$
	\sum\limits_{\Delta(k), \, \ell} = \sum\limits_{k=1}^\ell\sum\limits_{\stackrel{\ell_1+\ldots+\ell_k=\ell}{\ell_\nu \geq 1, \forall \nu}}.
	$

	Since $u \in \mathcal{F}_\mu$, there exist $\epsilon_0 >0$ and $C_5>0$ such that
	\begin{equation*}
		\sup_{\tau \in [0,2\pi]} | \partial_t^{\gamma-\ell} u_j(\tau)| \leq	C_5^{\gamma-\ell+1} [(\gamma-\ell)!]^{\theta} \exp \left(-\epsilon_0 j^{\frac{1}{2n\mu}} \right).
	\end{equation*}
	
	Moreover, applying Lemma \ref{lemma-exp-j} with $s=2n\mu$ and $p=m/2n$, we have
	\begin{equation*}
		j^{k m/2n} 
		\leq C_{\epsilon_0}^{k} (k !)^{m \mu} \exp \left(\frac{\epsilon_0}{2} j^{\frac{1}{2n\mu}} \right).
	\end{equation*}
	Hence, by setting $\widetilde{\sigma}=\max\{\theta, \sigma\}$ we get
	$$
	|\partial_t^{\gamma} \psi_j(t)|
	\leq C^{\gamma + 1} (\gamma!)^{\max\{\widetilde{\sigma}, m\mu-1\}} \exp\left(-\frac{\epsilon_0}{2} j^{\frac{1}{2n\mu}}
	\right),
	$$
	implying $\Psi u \in \mathcal{S}_{\max\{\widetilde{\sigma}, m\mu-1\},\mu}$ and that it is well defined.
	
	With similar arguments we may proof the same for  $\Psi^{-1}$.
	
\end{proof}

\begin{proposition}\label{LtoL-0}
	Let $\mathscr{E}_{L_{a_0},\mu}$ be  the space of admissible functions of operator $L_{a_0}$, that is, 
	the set of all $f \in \mathcal{S}_{\mu}$ such that
	$$
	\int_{0}^{2\pi}\exp\left(i\lambda_j \int_{0}^t (a_0+ ib(s))ds \right) f_j(t)dt = 0, 
	$$
	whenever $j\in \mathcal{Z}$. 		Then:
	
	\begin{enumerate}
		\item [(a)] $\Psi: \mathscr{E}_{L_{a_0},\mu} \to\mathscr{E}_{L,\mu}$ is an isomorphism;
		
		\item [(b)]  $L$ is $\mathcal{S}_{\mu}$-globally solvable if and only if the same is true for $L_{a_0}$;
		
		\item [(c)] If $\mathcal{Z} = \mathbb{N}$, then $L$ is  $\mathcal{S}_{\mu}$--globally solvable if and only if  
		$$
		L_b= D_t + ib(t)P(x,D_x)
		$$
		it is also $\mathcal{S}_{\mu}$-globally solvable.
		
	\end{enumerate}
	
\end{proposition}

\begin{proof}
	Part (a) is trivial. To verify (b), assume that $L$ is $\mathcal{S}_{\mu}$-globally solvable  and let  $f \in \mathscr{E}_{L_{a_0},\mu}$. There exists $u \in \mathscr{F}_\mu$ such that $Lu = \Psi (f)$, then it follows from   \eqref{conjugation} that $L_{a_0}[\Psi^{-1}(u)] = f$ and  the solvability of $L_{a_0}$.
	
	Viceversa, assume 	that $L_{a_0}$ is $\mathcal{S}_{\mu}$-globally solvable. Given $f\in \mathscr{E}_{L,\mu}$ there is $u \in \mathscr{F}_\mu$ such that $L_{a_0}u = \Psi^{-1} (f)$ implying
	$L[\Psi(u)] = f$  and  the global solvability of $L$.

	To verify (c) it is sufficient to observe that if $\mathcal{Z}= \N$, then functions
	\begin{equation*}
		\widetilde{\psi_j}(t) =  u_j(t) \exp\left(-i\lambda_j \int_{0}^{t}a(s)ds\right)
	\end{equation*}
	are $2\pi$-periodic for every $j \in \N$ and we may use
	$$
	\widetilde{\Psi} u = \sum_{j \in \N}\widetilde{\psi_j}(t) \varphi_j(x)
	$$
	to obtain a new conjugation formula instead of \eqref{conjugationformula}.
\end{proof}

\section{Proof of the main theorem} \label{Sec_proof_main_the}

In this section, we prove Theorem \ref{main_theorem}. We divide the proof in various steps.

\subsection{Sufficient conditions} The first step is Theorem \ref{suffientt_cond} where we show that  each of the conditions (a) and (b) in Theorem \ref{main_theorem} is sufficient for the global solvability of operator $L$.

In particular, we point out that in view of Proposition \ref{LtoL-0} it is equivalent to consider the operator 
$$
L_{a_0} = D_t + (a_0 + ib(t))P(x,D_x).
$$
Notice that if $b$ does not change sign and $c_0$ satisfies ($\mathscr{A}$), then either $b \equiv 0$ on $\T$ and $a_0$ 
satisfies ($\mathscr{A}$) or $b$ is not identically zero.
Then, the sufficiency in item (a) of Theorem \ref{main_theorem} is a direct consequence of the following result.

\vskip0.2cm
\begin{theorem}\label{suffientt_cond}
	Each of the following conditions is sufficient to guarantee the  $\mathcal{S}_{\mu}$-global solvability of operator $L_{a_0}$.

	\begin{enumerate}
		\item [(a)] $b\equiv 0$ and $a_0$ satisfies $(\mathscr{A})$;
		
		\item [(b)] $b\not\equiv 0$ and  does not change sign;
		
		\item [(c)] $b$ changes sign, $\mathcal{Z}^C$ is finite and $\Omega_{r}$ is connected for all $r \in \R$.
	\end{enumerate}

\end{theorem}

\begin{proof}
	Under condition  (a), $L_{a_0} = D_t + a_0P$ and we may apply  Theorem   \ref{GS-time-ind}.  Case (b) is a consequence of Theorem \ref{GH-time-dep} and Proposition \ref{GHimplGS}. To prove (c) we first assume $\mathcal{Z}=\N$. In this case, 
	$b_0=0$ and that by Proposition \ref{LtoL-0} it is  sufficient to consider 
	$$
	L_b = D_t + ib(t)P(x,D_x).
	$$
	
	Let $f \in \mathscr{E}_{L_b,\mu} = [\mbox{ker}(^t L_b)]^{\circ}$ be fixed. Since 
	\begin{equation*}
		\int_{0}^{2\pi} \exp\left(-\lambda_{j}\int_{0}^{t} b(r)dr\right)f_j(t) dt = 0, \forall j \in \mathbb{N}.
	\end{equation*}
	we obtain that the functions 
	\begin{equation}\label{solution-t_j}
		u_j(t) = i\exp\left(  \lambda_j\int_{0}^{t}  b(r) dr \right)   \int_{t_j}^{t}\exp\left(-\lambda_j \int_{0}^{s} b(r)dr \right) f_j(s)ds,
	\end{equation}
	for any $t_j \in \T$, 		belong to $\mathcal{G}^{\sigma}(\T)$  and solve
	\begin{equation}\label{ode_solution-t_j}
		D_t u_j(t) + \lambda_j ib(t) u_j(t) = f_j(t).
	\end{equation}

	We assume that $\lambda_{j}>0$ (the general case is analyzed in Remark \ref{lambda-I}).
	By defining 
	\begin{equation}\label{r_t}
		r_{t} =  \int_{0}^{t} b(r)dr, \ t \in \T
	\end{equation}
	we obtain that the set
	\begin{equation*}
		\Theta_{t}  = 
		\left\{s \in \T; \  \int_{0}^{s} b(\tau)d\tau \geq r_{t} \right\}
	\end{equation*}
	is connected. Therefore, if  $t_1 \in \T$ is given by 
	\begin{equation*}
		\int_{0}^{t_1}  b(r)dr = \max_{t \in \mathbb{T}} \left\{  \int_{0}^{t} b(r)dr\right\} 
	\end{equation*}
	then  $t, t_1 \in \Theta_{t}$ and there is an arc $\gamma_{t,t_1} \subset \Theta_{t}$ joining $t$ and $t_1$ implying 
	\begin{equation*}
		\int_{0}^{s} b(r)dr \geq 
		\int_{0}^{t}  b(r)dr, \ \forall s \in \gamma_{t,t_1}.
	\end{equation*}
	
	Then, 
	\begin{equation}\label{solution-arc-t_j}
		u_j(t) = i \int_{\gamma_{t,t_1}}\exp\left[\lambda_j\left(\int_{0}^{t}  b(r) dr-  \int_{0}^{s} b(r)dr\right) \right] f_j(s)ds
	\end{equation} 
	is a solution of \eqref{ode_solution-t_j}.  Moreover,  
	$$
	\lambda_j\left(\int_{0}^{t}  b(r) dr-  \int_{0}^{s} b(r)dr\right) \leq 0, \ \forall s,t \in \gamma_{t},
	$$
	implying that the exponential term in \eqref{solution-arc-t_j} is bounded by one. 
	
	The estimates for $u_j(t)$ are obtained by using Leibniz rule and Fa\`a di Bruno formula. Then,  we   get $\{u_{j}\}  \rightsquigarrow u \in   \mathscr{F}_{\mu}$.  
	
	For the general case $\mathcal{Z}\neq \N$, consider operator 
	$$
	L_{a_0} = D_t + (a_0 + ib(t))P,
	$$
	and $f \in \mathscr{E}_{L_{a_0},\mu}$. If $j \in \mathcal{Z}^{C}$ we define $u_j$ by expression \eqref{Solu-1}.  Since  $\mathcal{Z}^{C}$ is finite, estimates are unnecessary. 	Now, for $j \in \mathcal{Z}$ we replace \eqref{solution-arc-t_j} by 
	\begin{equation*}
		u_j(t) = i \int_{\gamma_{t,t_1}}\exp\left[\lambda_j\left(i a_0(s-t) + \int_{0}^{t}  b(r) dr -  \int_{0}^{s}b(r) dr\right) \right] f_j(s)ds
	\end{equation*} 
	and proceed as before.

\end{proof}

\begin{remark}\label{lambda-I}
	We point out that with a slight modification in the previous proof we can cover the general case where $\lambda_{j}$ is not positive. To see this, consider the sets
	$$
	\mathcal{W}_{+}=\{j \in \N; \, \lambda_{j}>0\} \ \textrm{ and } \ 
	\mathcal{W}_{-}=\{j \in \N; \, \lambda_{j}<0\}.	
	$$
	
	\begin{itemize}
		\item If $\mathcal{W}_{-}$ is finite: we use \eqref{Sol-6} as a solution of equation \eqref{ode_solution-t_j} when $j \in \mathcal{W}_{-}$ and \eqref{solution-arc-t_j} for $j \in \mathcal{W}_{+}$.
		
		\item If $\mathcal{W}_{+}$ is finite: we take \eqref{Sol-6} as a solution of  \eqref{ode_solution-t_j} if $j \in \mathcal{W}_{+}$, while in case $j \in \mathcal{W}_{-}$ we proceed as follows: let $r_{t}$ be as in \eqref{r_t} and consider the  connected set
		$$
		\widetilde{\Theta}_{t}  = 
		\left\{s \in \T; \  \int_{0}^{s} b(\tau)d\tau \leq r_{t} \right\}.
		$$
		
		Choosing  $t_1 \in \T$ as before we have
		$t, t_1 \in 	\widetilde{\Theta}_{t}$ and consequently there is an arc $\widetilde{\gamma}_{t,t_1} \subset \widetilde{\Theta}_{t} $ joining $t$ and $t_1$. In particular, 
		\begin{equation*}
			0 \leq 	\int_{0}^{t}  b(r)dr - \int_{0}^{s} b(r)dr, \ 
			\forall s \in \widetilde{\gamma}_{t,t_1}.
		\end{equation*}
		
		Hence, for $j \in \mathcal{W}_{-}$ we define
		\begin{equation}\label{solution-arc-t_j-negative}
			u_j(t) = i \int_{\widetilde{\gamma}_{t,t_1}}\exp\left[\lambda_j\left(\int_{0}^{t}  b(r) dr-  \int_{0}^{s} b(r)dr\right) \right] f_j(s)ds
		\end{equation} 
		Then,   
		$$
		\lambda_j\left(\int_{0}^{t}  b(r) dr-  \int_{0}^{s} b(r)dr\right) \leq 0, \ \forall s,t \in  \widetilde{\gamma}_{t,t_1},
		$$
		implying that the exponential term in \eqref{solution-arc-t_j-negative} is bounded by one.

		\item If both sets $\mathcal{W}_{+}$ and $\mathcal{W}_{-}$ are infinite: we use \eqref{solution-arc-t_j} for $j \in \mathcal{W}_{+}$ and \eqref{solution-arc-t_j-negative} for $j \in \mathcal{W}_{-}$.
	\end{itemize}
	
\end{remark}

\subsection{Necessary conditions} \label{sec_nec_cond} 

In this section, we investigate the necessity of each condition in Theorem \ref{main_theorem}.

\vskip0.2cm
\begin{proposition}\label{prop_a_0.b_0}
	If $b$ does not change sign and $c_0$ does not satisfy $(\mathscr{A})$, then $L$ is not $\mathcal{S}_{\mu}$-globally solvable.
\end{proposition}

\begin{proof}
	Assume $b\geq 0$. There is $\epsilon_0>0$ and an increasing sequence $j_\ell$ such that 	
	$$
	0<|1-e^{-2\pi i (a_0 + ib_0)\lambda_{j_{\ell}}}| <
	\exp\left\{-\epsilon_0 j_{\ell}^{1/2n\mu}\right\}, \ \ell \in \N.
	$$
	
	Let  $\phi \in C^{\infty}(\pi/2 - \delta, \pi/2+ \delta)$ be a cutoff function such that $\phi \equiv 1$ in a neighborhood of 
	$(\pi/2 - \delta/2, \pi/2+ \delta/2)$, where $\delta >0$ is such that $(\pi/2 - \delta, \pi/2+ \delta) \subset (0,\pi)$. Consider $f_{j_\ell}(t)$ a $2\pi$-periodic extension of
	$$
	\widetilde{f}_{j_\ell}(t) = \exp\left\{-\epsilon_0 j_{\ell}^{1/2n\mu}\right\}
	\exp\left[i\lambda_{j_\ell}a_0(\pi - t) \right]\phi(t), \ \ell \in \N.
	$$
	and set 
	\begin{equation*}
		f_j(t) = \left\{
		\begin{array}{l}
			0, \ j \neq {j_\ell},  \\
			f_{j_\ell}(t), j = {j_\ell}, 
		\end{array}
		\right.
	\end{equation*}
	for which $\{f_{j}\}  \rightsquigarrow f \in   \mathscr{F}_{\mu}$.
	
	Note that either $b_0 \neq 0,$ or  $b_0 = 0$ and $a_0\lambda_{j_{\ell}} \notin \Z$, for all $\ell \in \N$.
	In both cases, $j_{\ell} \notin \mathcal{Z}$ for all $\ell \in \N$ and $f_{j} \equiv 0$ for $j \in \mathcal{Z}$. It follows from Lemma \ref{charac_kern} that $f \in \mathscr{E}_{L,\mu}$.

	Now, we assume that $\lambda_{j_{\ell}} >0$ for all $\ell$ (See Remark \ref{lambda-II} for the general case).
	If $u \in \mathscr{F}_{\mu}$ is a solution of  $Lu = f$ we obtain
	$$
	u_{j_{\ell}}(t) =  \Gamma_{j_{\ell}} \int_{0}^{2\pi}
	\phi(t-s)
	\exp\left\{\lambda_{j_{\ell}} \int_{t-s}^{t} b(r) dr\right\} ds, 
	$$
	where
	$$
	\Gamma_{j_{\ell}} = i(1-e^{-2\pi i (a_0 + ib_0)\lambda_{j_{\ell}}})^{-1} \exp\left\{-\epsilon_0 j_{\ell}^{1/2n\mu}\right\}\exp\left\{i\lambda_{j_{\ell}}  a_0\right\}.
	$$
	
	In particular, for $t = \pi$,
	$$
	u_{j_{\ell}}(\pi) =  \Gamma_{j_{\ell}} \int_{\frac{\pi - \delta}{2}}^{\frac{\pi + \delta}{2}}
	\exp\left\{\lambda_{j_{\ell}}  \int_{\pi-s}^{\pi} b(r) dr\right\} ds. 
	$$
	
	Since $\lambda_{j_{\ell}}  \int_{\pi-s}^{\pi} b(r) dr\geq 0$,
	$$
	|u_{j_{\ell}}(\pi)| \geq  \delta  |1-e^{2\pi i (a_0 + ib_0)\lambda_{j_{\ell}}}|^{-1} \exp\left\{-\epsilon_0 j_{\ell}^{1/2n\mu}\right\} \geq \delta
	$$
	which is a contradiction. 		
\end{proof}

\begin{remark}\label{lambda-II}
	We can adapt the arguments in order to cover the case where $\lambda_{j_{\ell}}$ are not positive for every $\ell$. Indeed, 	consider
	$$
	\mathcal{W}_{+}=\{\ell \in \N; \, \lambda_{j_{\ell}}>0\} \ \textrm{ and } \ 
	\mathcal{W}_{-}=\{\ell \in \N; \, \lambda_{j_{\ell}}<0\}.	
	$$	
	
	\begin{itemize}
		\item If $\mathcal{W}_{+}$ is infinite, we redefine the sequence $f_j(t)$ as follows:
		\begin{equation*}
			f_j(t) = \left\{
			\begin{array}{l}
				0, \ j \neq {j_\ell}, \textrm{ or } \ell \in \mathcal{W}_{-}, \\
				f_{j_\ell}(t), j = {j_\ell}, \textrm{ and } \ell \in \mathcal{W}_{+}.
			\end{array}
			\right.
		\end{equation*}
		
		\item If $\mathcal{W}_{+}$ is finite, we then consider
		\begin{equation*}
			f_j(t) = \left\{
			\begin{array}{l}
				0, \ j \neq {j_\ell}, \textrm{ or } \ell \in \mathcal{W}_{+}, \\
				f_{j_\ell}(t), j = {j_\ell}, \textrm{ and } \ell \in \mathcal{W}_{-}
			\end{array}
			\right.
		\end{equation*}
		and  
		$$
		u_{j_{\ell}}(t) =  \widetilde{\Gamma}_{j_{\ell}} \int_{0}^{2\pi}
		\phi(t+s)
		\exp\left\{-\lambda_{j_{\ell}} \int_{t}^{t+s} b(r) dr\right\} ds, 
		$$
		where
		$$
		\widetilde{\Gamma}_{j_{\ell}} = i(e^{2\pi i (a_0 + ib_0)\lambda_{j_{\ell}}}-1)^{-1} \exp\left\{-\epsilon_0 j_{\ell}^{1/2n\mu}\right\}\exp\left\{i\lambda_{j_{\ell}}  a_0\right\}.
		$$
		Since $-\lambda_{j_{\ell}}  \int_{\pi-s}^{\pi} b(r) dr\geq 0$, we can proceed as before.
		
	\end{itemize}

\end{remark}

Next we analyze the case when $b$ changes sign. In this case the necessity in item (b) of Theorem \ref{main_theorem} can  be proved in several steps.

\vskip0.2cm
\begin{proposition}\label{prop1-nec-par}
	If $b$ changes sign and $b_0 \neq 0$, then 	$L$ is not $\mathcal{S}_{\mu}$-globally solvable.
\end{proposition}
\begin{proof}	
	Assume, without loss of generality, that $\lambda_{j} > 0$ for all $j \in \N$. We set
	$$
	\mathcal{H}(s,t) = \int_{t-s}^{t}b(\tau)d\tau, \ s,t \in [0,2\pi],
	$$
	and 
	$$
	M = \max_{s,t \in [0,2\pi]} H(s,t) = H(s^*, t^*). 
	$$
	
	Since 	$b$ changes sign we have $M>0$. In particular, we may assume $s^*\neq 0$, $s^* \neq t^*$ and
	$$
	0< s^*, t^*, \gamma <2\pi,
	$$	
	where $\gamma = t^* - s^*$.
	
	Consider $\delta>0$ such that $(\gamma - \delta, \gamma + \delta)\subset (0,t^*)$ and fix a cutoff function $\phi\in C_c^{\infty}(\gamma - \delta, \gamma + \delta)$ satisfying $0\leq \phi \leq 1$ and $\phi \equiv 1$ on a neighborhood of interval
	$(\gamma - \delta/2, \gamma + \delta/2)$.
	Let $f_j(t)$ be a $2\pi$-periodic extension of
	$$
	\widetilde{f}_j(t) = \exp\left(-\lambda_{j} M\right)
	\exp\left[i\lambda_{j}a_0 (t^*-t)\right]\phi(t), \ j \in \N.
	$$
	
	Since $M>0$, we get $\{f_{j}\}  \rightsquigarrow f \in   \mathscr{F}_{\mu}$. We claim that $f \in [\mbox{ker}(^tL)]^{\circ}$. Indeed, by $b_0\neq0$ we get $\mathcal{Z} = \emptyset$. Then, it follows from  Lemma \ref{charac_kern} that  $\omega \in \mbox{ker}(^tL)$ if and only if $\omega_j(t) \equiv 0$ for all $j \in \N$. 
	
	Next, we show that there is not $u \in \mathscr{F}_{\mu}$ such that $Lu = f$. To verify this we proceed by a contradiction argument: if there is such $u$  we would have
	$$
	u_j(t) =  \Theta_j \int_{t-\gamma-\delta}^{t-\gamma+\delta}
	\phi(t-s)
	\exp\left\{\lambda_{j} \left[-M + ia_0(t^*-t) + \int_{t-s}^{t} b(\tau) d\tau\right]\right\} ds, 
	$$
	where $\Theta_j = i\left(1 - e^{-  2 \pi i\lambda_j c_0}\right)^{-1}$. In particular, 
	\begin{equation*}
		u_j(t^*)  =  \Theta_j \int_{s^* -\delta}^{s^* +\delta}
		\phi(t^*-s)
		\exp\left\{-\lambda_{j} \left[M  -\mathcal{H}(s,t^*)\right]\right\} ds.
	\end{equation*}
	
	Since $b_0\neq 0$, there is a constant $0<C\leq |\Theta_j|$, for all $j$. Also, by the hypotheses on $\phi$:
	\begin{equation*}
		|u_j(t^*)| \geq C \int_{s^* -\delta/2}^{s^* +\delta/2}
		\exp\left\{-\lambda_{j} \left[M  -\mathcal{H}(s,t^*)\right]\right\} ds.
	\end{equation*}

	The function $\psi(s) = M  -\mathcal{H}(s,t^*)$ is positive with $\psi(s^*)=0$. Then,  $s^*$ is zero of even order and  there exists $k \in \N$ such that
	$$
	\psi^{(2k)}(s^*) \neq 0, \ \textrm{ and } \ 
	\psi^{(n)}(s^*)=0, \ n=1, \ldots, 2k-1.
	$$
	
	By Taylor's formula  we obtain $\eta \in (s^*-\delta/2, s^*+\delta/2)$
	satisfying
	$$
	\psi(s) = \dfrac{\psi^{(2k)}(s^*)}{(2k)!}(s-s^*)^{2k} + R_{2k+1}(s),
	$$
	for $s \in (s^*-\delta/2,s^*+\delta/2)$, where 
	$$
	R_{2k+1}(s) = \dfrac{\psi^{(2k+1)}(\eta)}{(2k+1)!}(s-s^*)^{2k+1}.
	$$
	
	Denoting
	$$
	C_1 = \left|\dfrac{\psi^{(2k+1)}(\eta)}{(2k+1)!}\right|
	\ \textrm{ and } \
	C_2 = \left|\dfrac{\psi^{(2k)}(s^*)}{(2k)!}\right|
	$$
	we get
	$$
	|\psi(s) - C_2(s-s^*)^{2k}| \leq C_1 |s-s^*|^{2k+1}
	$$
	and
	$$
	C_2 \left( 1 - \dfrac{C_3 \delta}{2} \right) (s-s^*)^{2k}
	\leq \psi(s) \leq
	C_2 \left( 1 + \dfrac{C_3 \delta}{2} \right) (s-s^*)^{2k},
	$$
	with $C_3 = C_1/C_2$. Therefore, there is $\beta>0$ such that
	\begin{equation*}
		|\psi(s)| \leq \beta (s-s^*)^{2k}, \ 	s \in (s^*-\delta/2,s^*+\delta/2).
	\end{equation*}

	Note that, 
	\begin{align*}
		\int_{s^* -\delta/2}^{s^* +\delta/2}
		\exp\left\{-\lambda_{j} \psi(s)\right\} ds  &  \geq  
		\int_{s^* -\delta/2}^{s^* +\delta/2}
		\exp\left\{-\lambda_{j} \beta (s-s^*)^{2k}\right\} ds \\
		& =
		\int_{-\frac{\delta}{2}}^{\frac{\delta}{2}}
		\exp\left\{-\lambda_{j}\beta s^{2k}\right\} ds \\
		& =
		\dfrac{1}{\lambda_{j}^{1/(2k)}}\dfrac{1}{\beta^{1/(2k)}}
		\int_{-\frac{\delta}{2}(\lambda_{j}\beta)^{1/(2k)}}
		^{\frac{\delta}{2}(\lambda_{j}\beta)^{1/(2k)}}
		\exp\left\{-s^{2k}\right\} ds.
	\end{align*}
	
	Note that
	$$
	\int_{-\infty}^{+\infty}\exp\left\{-s^{2k}\right\} ds = \dfrac{1}{k}\Gamma\left(\dfrac{1}{2k}\right)>0,
	$$
	where $\Gamma$ denotes the gamma function. Hence, there is $j_0 \in \N$ such that
	$$
	\int_{-\infty}^{+\infty}\exp\left\{-s^{2k}\right\} ds\geq\int_{-\frac{\delta}{2}(\lambda_{j}\beta)^{1/(2k)}}
	^{\frac{\delta}{2}(\lambda_{j}\beta)^{1/(2k)}}
	\exp\left\{-r^{2k}\right\} dr \geq \epsilon>0, \ \forall j\geq j_0,
	$$
	implying
	$$
	\lambda_{j}^{1/k}|u_j(t^*)| \geq 
	\lambda_{j}^{1/(2k)}\epsilon \beta^{-1/(2k)},  \forall j\geq j_0.
	$$

	Finally, given $N>0$ we obtain from \eqref{weyl}  some $j_1\geq j_0$ such that
	$$
	\lambda_{j}^{1/k}|u_j(t^*)| \geq 
	\lambda_{j}^{1/(2k)}\epsilon \beta^{-1/(2k)}\geq N,  \forall j\geq j_1.
	$$
	Hence, 
	$$
	\lim_{j \to \infty} \lambda_{j}^{1/k}|u_j(t^*)| = \infty,
	$$
	which is in contradiction with \eqref{deccoeff}.

\end{proof}

\begin{remark}\label{lambda-III}
	We point out that the assumption $\lambda_{j} > 0$ in the proof can be omitted. As in Remark \ref{lambda-I} we split the analysis considering the sets
	$$
	\mathcal{W}_{+}=\{j \in \N; \, \lambda_{j}>0\} \ \textrm{ and } \ 
	\mathcal{W}_{-}=\{j \in \N; \, \lambda_{j}<0\}.	
	$$
	
	Let us to show how to proceed when $\mathcal{W}_{-}$ is infinite. For $j\in \mathcal{W}_{-}$ we use 
	$$
	\mathcal{G}(s,t) = \int_{t}^{t+s}b(\tau)d\tau, \ s,t \in [0,2\pi],
	$$
	and take $f_j(t)$ as the $2\pi$-periodic extension of
	$$
	\widetilde{f}_j(t) = \exp\left(\lambda_{j} M\right)
	\exp\left[i\lambda_{j}a_0 (t^*-t)\right]\phi(t), \ j \in \N,
	$$
	where $M$ is the max of $\mathcal{G}(s,t)$ on $[0, 2\pi]$.
	
	Now we define the sequence $u_j(t)$ using expression \eqref{Solu-2}. Following the same procedure we get 
	\begin{equation*}
		u_j(t^*)  =  \widetilde{\Theta_j} \int_{s^* -\delta}^{s^* +\delta}
		\phi(t^*-s)
		\exp\left\{\lambda_{j} \left[M  -\mathcal{G}(s,t^*)\right]\right\} ds,
	\end{equation*}
	where $\widetilde{\Theta_j}=\left(e^{2 \pi i\lambda_j c_0} -1\right)^{-1}$ and 
	\begin{align*}
		\int_{s^* -\delta/2}^{s^* +\delta/2}
		\exp\left\{\lambda_{j} \psi(s)\right\} ds \geq 
		\dfrac{1}{(-\lambda_{j})^{1/(2k)}}\dfrac{1}{\beta^{1/(2k)}}
		\int_{-\frac{\delta}{2}((-\lambda_{j})\beta)^{1/(2k)}}
		^{\frac{\delta}{2}((-\lambda_{j})\beta)^{1/(2k)}}
		\exp\left\{-s^{2k}\right\} ds,
	\end{align*}
	implying
	$$
	\lim_{j \to \infty} (-\lambda_{j})^{1/k}|u_j(t^*)| = \infty.
	$$
	
\end{remark}

\begin{proposition}\label{prop_b_chang}
	If $b$ changes sign, $b_0 = 0$ and $\mathcal{Z}^C$ is an infinite set, then $L$ is not $\mathcal{S}_{\mu}$-globally solvable.
\end{proposition}

\begin{proof}
	Let	$\lambda_{j_\ell}$ be a subsequence such that $j_\ell > \ell$ and  $\lambda_{j_{\ell}}a_0 \notin \Z$. Consider $\mathcal{H}(s,t)$ and  $M$ be as in the proof of Proposition \ref{prop1-nec-par}. Also, we use $t^*, s^*, \gamma$ and $\phi(t)$ as before and set by 
	$f_{j_\ell}(t)$  a $2\pi$-periodic extension of
	$$
	\widetilde{f}_{j_\ell}(t) = \exp\left(-\lambda_{j_\ell} M\right)
	\exp\left[i\lambda_{j_\ell}a_0 (t^*-t)\right]\phi(t), \ \ell \in \N.
	$$
	Therefore, defining $f_{j}(t) \equiv 0$, if $j \neq j_\ell$, we get
	$\{f_{j}(t)\}  \rightsquigarrow f \in   [\mbox{ker}(^tL_{b})]^{\circ}$.
	
	If $Lu=f$, then
	$$
	u_{j_\ell}(t) =  \Theta_{j_\ell}\int_{t-\gamma-\delta}^{t-\gamma+\delta}
	\phi(t-s)
	\exp\left\{\lambda_{j_\ell} \left[-M + ia_0(t^*-t) + \int_{t-s}^{t} b(\tau) d\tau\right]\right\} ds, 
	$$
	where $\Theta_j = i\left(1 - e^{-  2 \pi i\lambda_j a_0}\right)^{-1}$, and
	\begin{equation*}
		u_{j_\ell}(t^*)  =  \Theta_{j_\ell} \int_{s^* -\delta}^{s^* +\delta}
		\phi(t^*-s)
		\exp\left\{-\lambda_{j_\ell} \left[M  -\mathcal{H}(s,t^*)\right]\right\} ds
	\end{equation*}
	implying
	\begin{equation*}
		|u_{j_\ell}(t^*)| \geq C \int_{s^* -\delta/2}^{s^* +\delta/2}
		\exp\left\{-\lambda_{j_\ell} \left[M  -\mathcal{H}(s,t^*)\right]\right\} ds.
	\end{equation*}
	
	Finally, we can proceed as in the proof of Proposition \ref{prop1-nec-par}.
\end{proof}

Combining Propositions \ref{prop1-nec-par} and \ref{prop_b_chang} we obtain that if $L$ is $\mathcal{S}_{\mu}$-globally solvable, then $b_0=0$ and $\mathcal{Z}^C$ is finite. To conclude the proof of the necessity in Theorem \ref{main_theorem} it is now sufficient to prove the next result.

\vskip0.2cm
\begin{theorem}\label{Theorem_nec_Omega}
	Assume that $b$ changes sign and $\mathcal{Z}$ is infinite. If   $\Omega_{r}$ is not connected for some $r \in \R$, then $L$ is not $\mathcal{S}_{\mu}$-globally solvable for every $\mu \geq \frac12$.
\end{theorem}

The proof of this theorem relies on the violation of the so-called H\"ormander condition, cf. \cite[Lemma 6.1.2]{H63}, which provides a necessary condition for solvability. We point out that this kind of approach is well known in the literature, as the reader may see in 
\cite{AlbanesePopiv, AlbaneseZanghirati, BDGK15,BDG18,BDG17,BCG21, Hou79, PETRONILHO2005}. Here we adapt the result in \cite{H63} to our functional setting as follows.

\begin{theorem}[Hörmander condition]
	Let $L$ be $\mathcal{S}_{\mu}$-globally solvable. Then, given $\sigma>1$, we obtain that for every $A>0$ and $B>0$ there exist a constant $C>0$ such that
	\begin{equation}\label{H_condition}
		\left | \int f v\right| \leq C \|f\|_{\sigma,\mu,A} \cdot \|^tLv\|_{\sigma,\mu,B}
	\end{equation}
	for all $f \in \mathscr{E}_{L,\mu}=[\mbox{ker}(^tL)]^{\circ}$ and $v \in \mathscr{F}_{\mu}$ such that 
	$^tLv \in \mathcal{S}_{\sigma,\mu,B}$, where we recall that 
	$$
	\|\varphi\|_{\sigma,\mu,B}=\sup_{M,\gamma \in \N} B^{-M-\gamma}\gamma!^{-\sigma} M^{-m\mu} \|P^M \partial_t^\gamma u\|_{L^2(\mathbb{T}\times \R^n)}.
	$$
	
\end{theorem}

\begin{proof}
	Consider the spaces
	$$
	\mathscr{G}_{\sigma,\mu}= \{  v \in \mathscr{F}_{\mu}: \,
	^tLv \in \mathcal{S}_{\sigma,\mu,B}\},
	$$
	and 
	$$
	\mathscr{O} = \dfrac{\mathscr{G}_{\sigma,\mu}}{\mathscr{G}_{\sigma,\mu} \cap ker(^tL)}. 
	$$
	Note that $[\mbox{ker}(^tL)]^{\circ}$ is a closed subspace of $\mathcal{S}_{\mu}$ and it can be endowed with the induced topology, while on $\mathscr{O}$ we consider the topology generated by  	the seminorms 
	$$
	\|[\phi]\|_{\mathscr{O}} = \|^tL \phi\|_{\sigma,\mu,k}, \ k \in \N_0.
	$$
	
	Now, let  $B: [\mbox{ker}(^tL)]^{\circ} \times \mathscr{O} \to \C$ be the bilinear form
	$$
	B(f,[\phi]) =  \int_{\T\times \R^n}f(t,x) \phi(t,x) dtdx.
	$$
	
	We observe that $B(\cdot, [\phi])$ is continuous on $[\mbox{ker}(^tL)]^{\circ}$, for every $[\phi] \in 	\mathscr{O}$. Let $f \in [\mbox{ker}(^tL)]^{\circ}$ be fixed and consider $u \in \mathscr{F}_{\mu}$ such that $Lu=f$. Hence, there exists $C=C(u)>0$
	\begin{align*}
		\left|	B(f,[\phi])\right| & = 	\left|\int_{\T\times \R^n}u(t,x) \, ^tL\phi(t,x) dtdx \right|
		= 	\left|<u, ^tL\phi>\right| 
		\leq C \|^tL \phi\|_{\sigma,\mu,k} 
	\end{align*}
	
	Therefore, we have: $[\mbox{ker}(^tL)]^{\circ}$ is a Fr\'echet space;  $\mathscr{O}$ is a metrizable topological vector space; 
	$B$ 	is separately continuous on $[\mbox{ker}(^tL)]^{\circ} \times \mathscr{O}$. Under these conditions, ti follows from   the Corollary of Theorem 34.1 in  \cite{treves} that the bilinear form $B$ is continuous. Then, \eqref{H_condition}  holds true.

\end{proof}

As in \cite{BDG18} (see also \cite[Lemma 2.2]{HouCardoso77}), we make use of the following result.

\vskip0.2cm
\begin{lemma}
	Suppose that there exist $r \in \R$  such that 
	$$
	\widetilde{\Omega}_{r} = \left\{t \in \T; \ -\int_{0}^{t} b(s)ds< r\right\}
	$$
	is not connected. Then, we can find a real number $r_0 <r$ such that $\widetilde{\Omega}_{r_0}$ has two connected components with disjoint closures. Moreover, we can construct  functions $f_0,v_0 \in \mathcal{G}^{\sigma}$, $\sigma >1$, satisfying the following conditions:
	$$
	\int_{0}^{2\pi}f_0(s)ds =0, \ \mbox{supp}(f_0) \cap \widetilde{\Omega}_{r_0} = \emptyset, \ \mbox{supp}(v'_0) \subset \widetilde{\Omega}_{r_0}
	$$
	and 
	$$
	\int_{0}^{2\pi}f_0(s)v_0(s)ds >0.
	$$
	
\end{lemma}

\noindent
\textit{Proof of Theorem \ref{Theorem_nec_Omega}.} Assume for a moment that $\mathcal{Z} =\N$, and consider  $L = D_t + ib(t)P$. 
Let   $r \in \R$ such that  $$\Omega_{r}=\left\{t \in \T : \int_0^t b(s)\, ds >r \right\}= \left\{t \in \T : -\int_0^t b(s)\, ds <-r \right\} $$ is not connected. By the previous Lemma, there is $r_0<-r$ such that
$$
\widetilde{\Omega}_{r_0} = \left\{t \in \T; \ -\int_{0}^{t} b(s)ds< r_0\right\}
$$
has two connected components with disjoint closures, and we can fix
$\epsilon>0$ such that
$$
M = \max_{t \in supp(v_0')}\left\{-\int_{0}^{t}b(s)ds\right\} < \epsilon <r_0,
$$

Define the sequences
$$
f_{\ell}(t,x) = \exp\left\{\lambda_{\ell} 
\left[ \epsilon  + \int_{0}^{t}b(s)ds \right] \right\} f_0(t)\varphi_{\ell}(x)  \in \mathcal{S}_{\sigma, 1/2}
$$
and 
$$
v_{\ell}(t,x) = \exp\left\{-\lambda_{\ell} 
\left[\epsilon + \int_{0}^{t}b(s)ds \right]  \right\} v_0(t) \overline{\varphi_{\ell}}(x) \in \mathcal{S}_{\sigma, 1/2} ,
$$
where $f_0, v_0$ are given in the Lemma. In particular, 
\begin{equation}\label{f_lv_l=0}
	\int_{\T \times \R^n}f_{\ell}(t,x)v_{\ell}(t,x)dtdx = 
	\int_{\T}f_{0}(t)v_{0}(t)dt >0,
\end{equation}
and 
$$
^tLv_{\ell}(t,x) = iv_{0}'(t) 
\exp\left\{-\lambda_\ell \left[\epsilon + \int_{0}^{t}b(s)ds\right]\right\} \overline{\varphi_{\ell}}(x) \in \mathcal{S}_{\sigma, 1/2}.
$$

For each $\ell \in \N$, the $x$-Fourier coefficients of $f_\ell(t,x)$ are given by
\begin{align*}
	f_{\ell,j}(t) & = \left(f_{\ell}(t,\cdot), \varphi_j(\cdot)\right)_{L^{2}(\R^n)} \\
	& = \exp\left\{\lambda_{\ell}\left[\epsilon  + \int_{0}^{t}b(s)ds\right] \right\}f_0(t) 
	\left( \varphi_{\ell}(\cdot), \varphi_j(\cdot)\right)_{L^{2}(\R^n)} \\
	& =
	\left\{
	\begin{array}{l}
		\exp\left\{\lambda_{j}\left[ \epsilon  + \int_{0}^{t}b(s)ds\right]  \right\}f_0(t)\ \ j=\ell, \\
		0, \ j\neq \ell.	
	\end{array}
	\right.
\end{align*}
Then,
$$
\int_{0}^{2\pi}\exp\left(i\lambda_j \int_{0}^{t}ib(s)ds\right) f_{\ell,j}(t)dt = 0, \ \forall j \in\N,
$$
implying 
$f_{\ell}(t,x) \in \mathscr{E}_{L,\mu}$, $\forall \ell \in \N$.

We claim that sequences $f_{\ell}$ and $v_{\ell}$ violate condition \eqref{H_condition}. Indeed, let $\alpha, \beta \in \Z^n_+$ and $\gamma \in \Z_+$. By defining 
$$
\mathcal{H}_{\ell}(t) = \exp\left\{\lambda_\ell \left[ \epsilon  +  \int_{0}^{t}b(s)ds\right] \right\}
$$
we get, for every $M, \gamma \in \N$:
\begin{equation*}
	\sup_{t \in \T}	\|P^M \partial_t^{\gamma} f_{\ell}(t,x)\|_{L^2( \R^n_x)} = |\lambda_\ell|^M 
	\sup_{t \in \T}|\partial_t^{\gamma}[\mathcal{H}_{\ell}(t) f_0(t)] |
\end{equation*}

Set $\Psi_{\ell,\gamma}(t) = \partial_t^{\gamma}[\mathcal{H}_{\ell}(t) f_0(t)]$. 
Since $f_0\in \mathcal{G}^{\sigma, h}$ it follows that
\begin{align*}
	|\Psi_{\ell,\gamma}(t)| & \leq  \sum_{\beta \leq \gamma}  \binom{\gamma}{\beta} \left |\partial_t^{\beta}\mathcal{H}_{\ell}(t)  \partial_t^{\gamma - \beta}f_0(t) \right| \\
	& \leq  \sum_{\beta \leq \gamma}  \binom{\gamma}{\beta} \left |\partial_t^{\beta}\mathcal{H}_{\ell}(t)\right|   C_1 h^{\gamma - \beta}[(\gamma-\beta)!]^{\sigma}, 
\end{align*}
and by  Fa\`a di Bruno formula
$$
\partial_t^{\beta}  \mathcal{H}_{\ell}(t) = \sum_{\Delta(\tau), \, \beta}
\frac{(\lambda_{\ell})^\tau}{\tau!}
\frac{\beta!}{\beta_1! \cdots \beta_\tau! }
\left(\prod_{\nu=1}^\tau
\partial_t^{\beta_\nu-1}b(t) \right) \mathcal{H}_{\ell}(t),
$$	
where
$
\sum\limits_{\Delta(\tau), \, \beta} = \sum\limits_{\tau=1}^\beta \sum\limits_{\stackrel{\beta_1+\ldots+\beta_\tau=\beta}{\beta_\nu \geq 1, \forall \nu}}. 
$
In view of 
$$
\left| \prod_{\nu=1}^\tau
\partial_t^{\beta_\nu-1}b(t)  \right| \leq C_2^{\beta-\tau+1}[(\beta -\tau)!]^\sigma
$$
we obtain
\begin{align*}
	|\Psi_{\ell,\gamma}(t)| & \leq  \sum_{\beta \leq \gamma}  \binom{\gamma}{\beta} \left |\partial_t^{\beta}\mathcal{H}_{\ell}(t)  \partial_t^{\gamma - \beta}f_0(t) \right| \\
	& \leq  \mathcal{H}_{\ell}(t) \sum_{\beta \leq \gamma}  \binom{\gamma}{\beta}   \sum_{\Delta(\tau), \, \beta}
	\frac{|\lambda_{\ell}|^\tau}{\tau!}\cdot
	\frac{\beta!}{\beta_1! \cdots \beta_\tau! }
	C_2^{\beta-\tau+1}[(\gamma -\tau)!]^\sigma   C_1 h^{\gamma - \beta} \\
	& \leq  \mathcal{H}_{\ell}(t) \sum_{\beta \leq \gamma}  \binom{\gamma}{\beta}   \sum_{\Delta(\tau), \, \beta}
	\frac{\rho^{\tau}}{\tau!}\cdot
	\frac{\beta!}{\beta_1! \cdots \beta_\tau! }
	C_2^{\beta-\tau+1}\ell^{\tau m/2n}[(\gamma -\tau)!]^\sigma   C_1 h^{\gamma - \beta},
\end{align*}
where $|\lambda_\ell| \leq \rho \ell^{\, m/2n}$.

We take $s=2n/m$ in Lemma \ref{lemma-exp-j}. Then, for any 
$\eta>0$ there is  $C_{\eta}>0$ such that 
$$ 
\ell^{\tau m/2n} \leq C_{\eta}^{\tau} \tau ! \exp(\eta \ell^{m/2n}), 
$$
implying
\begin{equation*}
	|\partial_t^{\gamma}[\mathcal{H}_{\ell}(t) f_0(t)] |
	\leq C_\eta^{\gamma+1}	\gamma !^\sigma \mathcal{H}_{\ell}(t) \exp(\eta \ell^{m/2n})		
\end{equation*}
for $\ell$ large enough.

Now, applying again Lemma \ref{lemma-exp-j} we obtain that for every $\eta >0$ there exists $C_\eta>0$ such that
$$
|\lambda_\ell^M| \leq C_\eta^{M+1} M!^{m\mu} \exp (\eta \ell^{\frac1{2n \mu}}).
$$

Hence, we can extimate the norm \eqref{secondnorm} of $f_\ell$ as follows:
$$
\|f_{\ell}\|_{\sigma,\mu,C_\eta} \leq C_\eta^2 \exp(2\eta \ell^{1/2n\mu}) \sup_{t \in supp (f_0)}\{\mathcal{H}_{\ell}(t)\}.
$$

By a similar procedure:	
$$
\|^tLv_{\ell}\|_{\widetilde{\sigma},\mu,C'} \leq C^2_\eta \exp(2\eta  \ell^{1/2n\mu}) \sup_{t \in supp (v_0')}\{\mathcal{H}_{\ell}^{-1}(t)\} 
$$
Now we recall that
$$
\rho^*\ell^{m/2n} \leq |\lambda_{\ell}|
\leq \rho\ell^{m/2n}, \qquad \ell \, \, \textit{large}, 
$$	
for some positive constants $\rho, \rho*$. If $\lambda_{\ell}>0$, then
$$
\sup_{t \in supp (f_0)}\{\mathcal{H}_{\ell}(t)\} \leq  \exp (c_1\rho^* \ell^{m/2n}),
$$
and
$$
\sup_{t \in supp (v_0')}\{\mathcal{H}_{\ell}^{-1}(t)\} \leq \exp (c_2\rho^* \ell^{m/2n}),
$$
where 
$$
c_1 = \max_{t \in supp(f_0)}
\left\{\epsilon + \int_{0}^{t}b(s)ds \right\} <0,
$$
and 
$$
c_2 = \max_{t \in supp(v_0')}\left\{- \left[\epsilon + \int_{0}^{t}b(s)ds\right]\right\} <0.
$$

Therefore, 
\begin{equation}\label{f_l,tLv}
	\|f_{\ell}\|_{\sigma,\mu,C} \cdot  \|^tLv_{\ell}\|_{\sigma,\mu,C'} 
	\leq \ell^{m/n} \exp\left[\rho^*(c_1 + c_2)\ell^{m/2n} + 4 \eta\ell^{1/2n\mu}\right].
\end{equation}

Then, since $\mu \geq \frac12 \geq \frac1{m},$  choosing  $\eta$ such that 	$\rho^*(c_1 + c_2) + 4\eta<0$ and obtain that		
$$
\lim\limits_{\ell \to \infty} \left[\|f_{\ell}\|_{\sigma,\mu,C} \cdot \|^tLv_{\ell}\|_{\sigma,\mu,C'}\right]  = 0.
$$

On the  other hand, if $\lambda_{\ell}<0$, then the right side in
\eqref{f_l,tLv} it becomes
$$
\ell^{m/n} \exp\left[\rho(c_1 + c_2)\ell^{m/2n} + 4 \eta\ell^{1/2n\mu}\right]
$$
and again we obtain the same contradiction.

It follows  from  \eqref{f_lv_l=0} that
condition \eqref{H_condition} can not be fulfilled implying that
$L$ is not $\mathcal{S}_{\mu}$-globally solvable, for
$\mu \geq \frac12$.

Finally, in the general case $\mathcal{Z} \neq \N$ the proof is given by a slight modification in the previous arguments. Indeed, we may consider the operator $L = D_t +  (a_0+ib(t))P$, and sequences
$$
\tilde{f}_{\ell}(t,x) = 
\left\{
\begin{array}{l}
	\exp(i\lambda_\ell a_0t)f_{\ell}(t,x), \ \textrm{ if } \ \ell \in  \mathcal{Z},\\
	0, \ \textrm{ if } \ \ell \notin \mathcal{Z},
\end{array}
\right.
$$
$$
\tilde{v}_{\ell}(t,x) = 
\left\{
\begin{array}{l}
	\exp(-i\lambda_\ell a_0t)	v_{\ell}(t,x), \ \textrm{ if } \ \ell \in \mathcal{Z}, \\
	0, \ \textrm{ if } \  \ell \notin \mathcal{Z}.
\end{array}
\right.
$$
instead of $f_\ell$ and $v_\ell$.

\qed

\noindent
\textbf{Acknowledgements}
\vskip0.2cm
\noindent
The research of Fernando de \'Avila Silva was supported in part by National Council for Scientific and Technological Development 
- CNPq - (grants 423458/2021-3, 402159/2022-5, 200295/2023-3, and 305630/2022-9). Part of this work was developed while  Fernando was a visitor to the Department of Mathematics at the University of Turin. He is very grateful for all the support and hospitality.

\end{document}